\documentclass[oneside,11pt]{amsart}
\usepackage{amsmath, amsfonts,amsthm,times,graphics}

\usepackage[active]{srcltx}
 \makeatletter
\renewcommand*\subjclass[2][2000]{%
  \def\@subjclass{#2}%
  \@ifundefined{subjclassname@#1}{%
    \ClassWarning{\@classname}{Unknown edition (#1) of Mathematics
      Subject Classification; using '1991'.}%
  }{%
    \@xp\let\@xp\subjclassname\csname subjclassname@#1\endcsname
  }%
}
 \makeatother
\newtheorem{theorem}{Theorem}[section]
\newtheorem{lemma}[theorem]{Lemma}
\newtheorem{corollary}[theorem]{Corollary}
\newtheorem{proposition}[theorem]{Proposition}
\theoremstyle{definition}
\newtheorem{definition}[theorem]{Definition}

\newtheorem{remark}[theorem]{Remark}
\numberwithin{equation}{section}
\newcommand{\abs}[1]{\lvert#1\rvert}
 \makeatletter
\renewcommand*\subjclass[2][2000]{%
  \def\@subjclass{#2}%
  \@ifundefined{subjclassname@#1}{%
    \ClassWarning{\@classname}{Unknown edition (#1) of Mathematics
      Subject Classification; using '1991'.}%
  }{%
    \@xp\let\@xp\subjclassname\csname subjclassname@#1\endcsname
  }%
}

\newcommand{\R}{{\mathbb{R}}}

\renewcommand{\L}{{\rm L}}

\newcommand{\eps}{{\varepsilon}}

\def\NABLA#1{{\mathop{\nabla\kern-.5ex\lower1ex\hbox{$#1$}}}}
\def\Nabla#1{\nabla\kern-.5ex{}_{#1}}
\def\Tabla#1{\Tilde\nabla\kern-.5ex{}_{#1}}
\def\abs#1{\mathopen|#1\mathclose|}

\renewcommand{\Tilde}{\widetilde}

 \makeatother

\begin{document}

\title[{Invertible harmonic mappings beyond Kneser and q.c. harmonic mappings}]{Invertible harmonic mappings beyond Kneser theorem and quasiconformal harmonic mappings}
\author{David Kalaj}
\address{University of Montenegro, faculty of natural sciences and mathematics,
Cetinjski put b.b. 81000, Podgorica, Montenegro}
\email{davidk@t-com.me} \subjclass {Primary 30C55, Secondary 31C05}

\begin{abstract}
In this paper we extend Rado-Choquet-Kneser theorem for the mappings
with Lipschitz boundary data and essentially positive Jacobian at
the boundary without restriction on the convexity of image domain.
The proof is based on a recent extension of Rado-Choquet-Kneser
theorem by Alessandrini and Nesi \cite{ale} and it is used an
approximation schema. Some applications for the family of
quasiconformal harmonic mappings between Jordan domains are given.

\end{abstract}
\maketitle 
\section{Introduction and statement of the main result}
Harmonic mappings in the plane are univalent complex-valued harmonic
functions of a complex variable. Conformal mappings are a special
case where the real and imaginary parts are conjugate harmonic
functions, satisfying the Cauchy-Riemann equations. Harmonic
mappings were studied classically by differential geometers because
they provide isothermal (or conformal) coordinates for minimal
surfaces. More recently they have been actively investigated by
complex analysts as generalizations of univalent analytic functions,
or conformal mappings. For the background to this theory we refer to
the book of Duren \cite{dure}. If $w$ is a univalent complex-valued
harmonic functions, then by Lewy's theorem (see \cite{l}), $w$ has a
non-vanishing Jacobian and consequently, according to the inverse
mapping theorem, $w$ is a diffeomorphism.
Moreover, if $w$ is a harmonic mapping of the unit disk $\mathbf U$
onto a convex Jordan domain $\Omega$, mapping the boundary $\mathbf
T=\partial \mathbf U$ onto $\partial \Omega$ homeomorphically, then
$w$ is a diffeomorphism. This is celebrated theorem of Rado, Kneser
and Choquet (\cite{knes}). This theorem has been extended in various
directions (see for example \cite{jj}, \cite{an}, \cite{sy} and
\cite{sf}). One of the recent extensions is the following
proposition, due to Nesi and Alessandrini, which is one of the main
tools in proving our main result.


\begin{proposition}\label{ales}\cite{ale} Let $F: \mathbf T\to
\gamma\subset \mathbf C$ be an orientation preserving diffeomorphism
of class $C^1$ onto a simple closed curve of the complex plane
$\mathbf C$. Let $D$ be a bounded domain such that $\partial D
=\gamma$. Let $w=P[F]\in C^1(\overline{\mathbf U};\mathbf C),$ where
$P[f]$ is the Poisson extension of $F$. The mapping $w$ is a
diffeomorphism of $\overline{\mathbf U}$ onto $\overline{D}$ if and
only if
\begin{equation}\label{use} J_w(e^{it})
> 0\text{ everywhere on}\  \mathbf T,\end{equation} where $J_w(e^{it})
:=\lim_{r\to 1^-}J_w(re^{it})$, and $J_w(re^{it})$ is the Jacobian
of $w$ at $re^{it}$.
\end{proposition}
%
%

In this paper we generalize Rado-Kneser-Choquet theorem as follows.

\begin{theorem}[The main result]\label{ma} Let $F: \mathbf T\to
\gamma\subset \mathbf C$ be an orientation preserving Lipschitz weak
homeomorphism of the unit circle $\mathbf T$ onto a $C^{1,\alpha}$
smooth Jordan curve. Let $D$ be a bounded domain such that $\partial
D =\gamma$. Then $J_w(e^{it})/|F'(t)|$ exists a.e. in $\mathbf T$
and has a continuous extension $T_w(e^{it})$ to $\mathbf T$. If
\begin{equation}\label{use12}
T_w(e^{it})>0\ \ \text{  everywhere on}\ \ \mathbf T,\end{equation}
then the mapping $w=P[F]$ is a diffeomorphism of $\mathbf U$ onto
$D$.
\end{theorem}

 In order to compare this statement with Kneser's Theorem, it is
worth noticing that when $D$ is convex, then by Remark~\ref{rere}
the condition \eqref{use12} is automatically satisfied.

It follows from Theorem~\ref{ma} that under its conditions, the
Jacobian $J_w$ of $w$ has continuous extension to the boundary
provided that $F\in C^1(\mathbf T)$ and it should be noticed that
this \underline{does not} mean that the partial derivatives of $w$
have necessarily a continuous extension to the boundary (see e.g.
\cite{om} for a counterexample).

Note that we do not have any restriction on convexity of image
domain in Theorem~\ref{ma} which is  proved in section~3.

Using this theorem, in section~4 we characterize all quasiconformal
harmonic mappings between the unit disk $\mathbf U$ and a smooth
Jordan domain $D$, in terms of boundary data (see
Theorem~\ref{convex1}) which could be considered as a variation of
Proposition~\ref{ales}. 
%
%
\section{Preliminaries}
\subsection{Arc length parameterization of a Jordan curve} Suppose
that $\gamma$ is a rectifiable Jordan curve in the complex plane
$\mathbf C$. Denote by $l$ the length of $\gamma$ and let
$g:[0,l]\mapsto \gamma$ be an arc length parameterization of
$\gamma$, i.e. a parameterization satisfying the condition:
$$| g'(s)|=1 \text{ for all $s\in [0,l]$}.$$
{We will say that  $\gamma$ is of class $C^{1,\alpha}$, $0<\alpha\le
1$, if $g$  is of class   $C^1$    and
$$\sup_{t,s}\frac{|g'(t)-g'(s)|}{|t-s|^{\alpha}}<\infty.$$

\begin{definition} Let $l=|\gamma|$.
We will say that a surjective function $F=g\circ f:\mathbf T\to
\gamma$ is a \emph{weak homeomorphism}, if $f:[0,2\pi]\to[0,l]$ is a
nondecreasing surjective function.
\end{definition}

\begin{definition}
Let $f:[a,b]\to \mathbf C$ be a continuous function. The modulus of
continuity of $f$ is
   $$\omega(t)= \omega_f(t) = \sup_{|x-y|\le t} |f(x)-f(y)|. \,$$
The function $f$ is called Dini continuous if
\begin{equation}\label{didi}\int_{0^+} \frac{\omega_f(t)}{t}\,dt < \infty.\end{equation} Here
$\int_{0^+}:=\int_{0}^k$ for some positive constant $k$.  A smooth
Jordan curve $\gamma$ with the length $l=|\gamma|$, is said to be
Dini smooth if $g'$ is Dini continuous. Observe that every smooth
$C^{1,\alpha}$ Jordan curve is Dini smooth.
\end{definition}
Let
\begin{equation}\label{ker}K(s,t)=\text{Re}\,[\overline{(g(t)-g(s))}\cdot i
g'(s)]\end{equation} be a function defined on $[0,l]\times[0,l]$. By
$K(s\pm l, t\pm l)=K(s,t)$ we extend it on $\mathbf R\times \mathbf
R$. Note that $ig'(s)$ is the inner unit normal vector of $\gamma$
at $g(s)$ and therefore, if $\gamma$ is convex then
\begin{equation}\label{convexkernel}K(s,t)\ge 0\text{ for every $s$
and $t$}.\end{equation} Suppose now that $F:\mathbf R\mapsto \gamma$
is an arbitrary $2\pi$ periodic Lipschitz function such that
$F|_{[0,2\pi)}:[0,2\pi)\mapsto \gamma$ is an orientation preserving
bijective function. { Then there exists an increasing continuous
function $f:[0,2\pi]\mapsto [0,l]$ such that
\begin{equation}\label{fgs}F(\tau)=g(f(\tau)).\end{equation}}
In the remainder of this paper we will identify $[0,2\pi)$ with the
unit circle $\mathbf T$,  and $F(s)$ with $ F(e^{is})$. In view of
the previous convention we have for a.e. $e^{i\tau}\in \mathbf T$
that
$$F'(\tau)=g'(f(\tau))\cdot f'(\tau),$$ and therefore
$$|F'(\tau)|=|g'(f(\tau))|\cdot |f'(\tau)|=f'(\tau).$$
Along with the function $K$ we will also consider the function $K_F$
defined by
$$K_F(t,\tau)=\text{Re}\,[\overline{(F(t)-F(\tau))}\cdot
iF'(\tau)].$$ It is easy to see that
\begin{equation}\label{kg}K_F(t,\tau)=f'(\tau)K(f(t), f(\tau)).
\end{equation}

\begin{lemma}
If $\gamma$ is Dini smooth, and $\omega$ is modulus of continuity of
$g'$, then
\begin{equation}\label{oh}|K(s,t)|\le
\int_0^{\min\{|s-t|,l-|s-t|\}}\omega(\tau)d\tau.\end{equation}
\end{lemma}

\begin{proof}
Note that
\[\begin{split}K(s,t)&=\text{Re}[\overline{(g(t)-g(s))}\cdot i g'(s)]\\&=
\text{Re}\left[\overline{(g(t)-g(s))}\cdot i
\left(g'(s)-\dfrac{g(t)-g(s)}{t-s}\right)\right],\end{split}\] and
$$g'(s)-\dfrac{g(t)-g(s)}{t-s}=\int_s^t\frac{g'(s)-g'(\tau)}{t-s}d\tau.$$
Therefore
\[\begin{split}\left|g'(s)-\dfrac{g(t)-g(s)}{t-s}\right|&\le\int_s^t
\frac{|g'(s)-g'(\tau)|}{t-s}d\tau\\ &\le \int_s^t
\frac{\omega(\tau-s)}{t-s}d\tau
\\&=\frac{1}{t-s}{\int_0^{t-s}\omega(\tau)d\tau}.
\end{split}\] On the other hand $$|\overline{g(t)-g(s)}|\le
\sup_{s\le x\le t}|g'(x)| (t-s)=(t-s).$$ It follows that

\begin{equation}\label{ohoh}|K(s,t)|\le
\int_0^{|s-t|}\omega(\tau)d\tau.\end{equation} Since $K(s\pm l, t\pm
l) = K(s,t)$ according to \eqref{ohoh} we obtain \eqref{oh}.
\end{proof}

\begin{lemma}
If $\omega:[0,l]\to [0,\infty)$, $\omega(0)=0$, is a bounded
function satisfying $\int_{0^+} \omega(x) dx/ x<\infty$, then for
every constant $a$, we have $\int_{0^+} \omega(a x) dx/ x<\infty$.
Next for every $0<y\le l$ holds the following formula:
\begin{equation}\label{goodd}\int_{0+}^y\frac{1}{x^2}\int_0^x\omega(a t) dt dx =
\int_{0+}^y\frac{\omega(ax)}{x}-\frac{\omega(ax)}{y}dx.\end{equation}
\end{lemma}
\begin{proof} The first statement of the lemma is immediate.
Taking the substitutions $u = \int_0^x\omega(a t) dt $ and $dv =
x^{-2} dx$, and using the fact that $$\lim_{\alpha \to
0}\frac{\int_0^{\alpha}\omega(a t) dt}{\alpha} = \omega(0) = 0$$ we
obtain:
\[\begin{split}\int_{0+}^y\frac{1}{x^2}\int_0^x\omega(a t) dt dx
&=\lim_{\alpha\to 0+}\int_{\alpha}^y\frac{1}{x^2}\int_0^x\omega(a t)
dt dx\\& =-\lim_{\alpha\to 0+}\left.\frac{\int_0^x\omega(a t)
dt}{x}\right|_{\alpha}^{y} +\lim_{\alpha\to
0+}\int_{\alpha}^{y}\frac{\omega(a
x)}{x}dx\\&=\int_{0+}^y\frac{\omega(ax)}{x}-\frac{\omega(ax)}{y}dx.\end{split}\]

\end{proof}
A function $\varphi:A\to B$ is called  $(\ell,\mathcal{L)}$
bi-Lipschitz, where $0<\ell<\mathcal{L}<\infty$, if $\ell|x-y|\le
|\varphi(x)-\varphi(y)|\le\mathcal{L} |x-y|$ for $x,y\in A$.
\begin{lemma}\label{dri}
If $\varphi:\mathbf R\to \mathbf R$ is a $(\ell,\mathcal{L})$
bi-Lipschitz mapping ($\mathcal{L}$ Lipschitz weak homeomorphism),
such that $\varphi(x+a) = \varphi(x) + b$ for some $a$ and $b$ and
every $x$, then there exist a sequence of  $(\ell,\mathcal{L})$
bi-Lipschitz diffeomorphisms (respectively a sequence of
diffeomorphisms) $\varphi_n:\mathbf R\to \mathbf R$ such that
$\varphi_n$ converges uniformly to $\varphi$, and $\varphi_n(x+a) =
\varphi_n(x)+b$.
\end{lemma}

\begin{proof}

We introduce appropriate mollifiers: fix a smooth function
$\rho:\R\to[0,1]$ which is compactly supported on the interval
$(-1,1)$ and satisfies $\int_\R\rho=1$. For $\eps=1/n$ consider the
mollifier
\begin{equation}\label{mol}
     \rho_\eps(t):=\frac{1}{\eps}\,
     \rho\left(\frac{t}{\eps}\right).
\end{equation}
It is compactly supported in the interval $(-\eps,\eps)$ and
satisfies $\int_\R\rho_\eps=1$. Define
$$\varphi_\varepsilon(x)=\varphi\ast \rho_\eps = \int_{\mathbf R} \varphi(y)
\frac{1}{\varepsilon}\rho\left(\frac{x-y}{\varepsilon}\right)dy=\int_{\mathbf
R} \varphi(x-\varepsilon z)\rho(z)dz.$$ Then
$$\varphi'_\varepsilon(x) =\int_{\mathbf R} \varphi'(x-\varepsilon
z)\rho(z)dz.$$ It follows that
$$\ell\int_{\mathbf R} \rho(z)dz=\ell\le |\varphi'_\varepsilon(x)|\le \mathcal{L}\int_{\mathbf R} \rho(z)dz=\mathcal{L}.$$
The fact that $\varphi_\varepsilon(x)$ converges uniformly to
$\varphi$ follows by Arzela-Ascoli theorem.

To prove the case, when $\varphi$ is a $\mathcal{L}$ Lipschitz weak
homeomorphism, we make use of the following simple fact. Since
$\varphi$ is $\mathcal{L}$ -Lipschitz, then
$$\varphi_m(x)=\frac{mb}{mb+a}(\varphi(x)+x/m)$$ is
$(\ell_m,\mathcal{L}_m)$ bi-Lipschitz, for some
$\ell_m,\mathcal{L}_m$, with $\varphi_m(x+a)=\varphi_m(x)+b$, and
$\varphi_m$ converges uniformly to $\varphi$. By the previous case,
we can choose a diffeomorphism
\begin{equation}\label{aar}\psi_{m}=\varphi_m\ast \rho_{\varepsilon_m}=\frac{m
b}{m
b+a}\left(\varphi\ast\rho_{\varepsilon_m}+\frac{x}{m}\right),\end{equation}
such that $\|\psi_m-\varphi_m\|_\infty\le 1/m$. Thus
$$\lim_{n\to\infty}\|\psi_n-\varphi\|_\infty=0.$$
The proof is completed.
\end{proof}

\subsection{Harmonic functions and Poisson integral}
The function
$$P(r,t)=\frac{1-r^2}{2\pi (1-2r\cos t+r^2)},\ \ \ 0\le r<1,\ \ t\in[0,2\pi]$$ is called the Poisson
kernel. The Poisson integral of a complex function $F\in L^1(\mathbf
T)$ is a complex harmonic function given by
\begin{equation}\label{e:POISSON}
w(z)=u(z)+iv(z)=P[F](z)=\int_0^{2\pi}P(r,t-\tau)F(e^{it})dt,
\end{equation}
where $z=re^{i\tau}\in \mathbf U$. We refer to the book of Axler,
Bourdon and Ramey \cite{ABR} for good setting of harmonic functions.

The Hilbert transformation of a function $\chi\in \L^1(\mathbf T)$
is defined by the formula
$$\tilde\chi(\tau)=H(\chi)(\tau)=-\frac 1\pi
\int_{0+}^\pi\frac{\chi(\tau+t)-\chi(\tau-t)}{2\tan(t/2)}\mathrm
dt.$$ This integral is improper and converges for a.e.
$\tau\in[0,2\pi]$; this and other facts concerning the operator $H$
used in this paper can be found in the book of Zygmund
\cite[Chapter~VII]{ZY}. If $f$ is a harmonic function then a
harmonic function $\tilde f$ is called the harmonic conjugate of $f$
if $f+i\tilde f$ is an analytic function. Let $\chi, \tilde\chi\in
L^1(\mathbf T)$. Then
\begin{equation}\label{lit}P[\tilde \chi]=\widetilde
{P[\chi]},\end{equation} where $\tilde k(z)$ is the harmonic
conjugate of $k(z)$ (see e.g. \cite[Theorem~6.1.3]{lib}).

Assume that $z=x+iy=r e^{i\tau}\in\mathbf U$. The complex
derivatives of a differential mapping $w:\mathbf U\to \mathbf C$ are
defined as follows:
$$w_z=\frac{1}{2}\left(w_x+\frac{1}{i} w_y\right)$$ and $$w_{\bar z}=\frac{1}{2}\left(w_x-\frac{1}{i} w_y\right).$$
The derivatives of $w$  in polar coordinates can be expressed as
$$w_\tau(z):=\frac{\partial w(z)}{\partial
\tau}=i(zw_z-\overline{z}w_{\bar z})$$ and $$w_r(z):=\frac{\partial
w(z)}{\partial r}=(e^{i\tau}w_z+e^{-i\tau}w_{\bar z}).$$ The
Jacobian determinant of $w$ is expressed in polar coordinates as
\begin{equation}\label{jaco}J_w(z)=|w_z|^2-|w_{\bar
z}|^2=\frac{1}{r}\mathrm{Im}(h_\tau\overline{h_r})=\frac{1}{r}\mathrm{Re}(i
h_r\overline{h}_\tau).\end{equation}
Assume that $w=P[F](z)$ is a
harmonic function defined on the unit disk $\mathbf U$. Then there
exist two analytic functions $h$ and $k$ defined in the unit disk
such that $w=h+\overline{k}$. Moreover $w_\tau=i(z h'(z)-\bar z
\overline{k'(z)})$ is a harmonic function and $rw_r=zh'(z)+\bar z
\overline{k'(z)}$ is its harmonic conjugate.

It follows from \eqref{e:POISSON} that $w_\tau$ equals the
Poisson--Stieltjes integral of $F'$:
\[\begin{split}w_\tau(re^{i\tau}) &=\int_0^{2\pi}\partial_\tau
P(r,\tau-t)F(t)dt\\&=-\int_0^{2\pi}\partial_t
P(r,\tau-t)F(t)dt\\&=-\int_0^{2\pi}\partial_t
P(r,\tau-t)F(t)dt\\&=-P(r,\tau-t)F(t)|_{t=0}^{2\pi}+\int_0^{2\pi}
P(r,\tau-t)dF(t)\\&= \int_0^{2\pi}P(r,\tau-t)dF(t).\end{split}\]
Hence, by Fatou's theorem, the radial limits of $w_\tau$ exist a.e.
and
$$\lim_{r\to 1^-} w_\tau(re^{i\tau}) = F_0'(\tau),\ \  a.e.,$$ where
$F_0$ is the absolutely continuous part of $F$.

As $rw_r$ is the harmonic conjugate of $w_\tau$, it turns out that
if $F$ is absolutely continuous, then
\begin{equation}\label{amp}\lim_{r\to 1^-}w_r(re^{i\tau}) =
H(F')(\tau)\,\, (a.e.),\end{equation} and
\begin{equation}\label{pam}\lim_{r\to 1^-} w_\tau(re^{i\tau}) =
F'(\tau)\ \ \ (a.e).\end{equation}

\section{The proof of the main theorem}

The aim of this chapter is to prove Theorem~\ref{ma}. We will
construct a suitable sequence $w_n$ of univalent harmonic mappings,
converging almost uniformly to $w=P[F]$.  In order to do so, we will
mollify the boundary function $F$, by a sequence of diffeomorphism
$F_n$ and take the Poisson extension $w_n=P[F_n]$. We will show that
under the condition of Theorem~\ref{ma} for large $n$, $w_n$
satisfies the conditions of theorem of Alessandrini and Nesi. By a
result of Hengartner and Schober \cite{hsc}, the limit function $w$
of a locally uniformly convergent sequence of univalent harmonic
mappings $w_n$ is univalent, providing that $F$ is a surjective
mapping.

We begin by the following lemma.

\begin{lemma}\label{form} Let $\gamma$ be a Dini smooth Jordan curve,  denote by $g$ its
arc-length parameterization and assume that $F(t)=g(f(t))$ is a
Lipschitz weak homeomorphism from the unit circle onto $\gamma$. If
$w(z)=u(z)+iv(z)=P[F](z)$ is the Poisson extension of $F$, then for
almost every $\tau\in [0,2\pi]$ exists the limit
$$J_w(e^{i\tau}) :=\lim_{r\to 1^-} J_w(re^{i\tau})$$ and there holds the
formula
\begin{equation}\label{jacfor}\begin{split}J_w(e^{i\tau})&= f'(\tau)\int_0^{2\pi}
\frac{\mathrm{Re}\,[\overline{(g(f(t))-g(f(\tau)))}\cdot i
g'(f(\tau))]}{2\sin^2\frac{t-\tau}{2}}dt.\end{split}\end{equation}
\end{lemma}

\begin{proof} Let $z=r e^{i\tau}$. Since $F$ is Lipschitz it is
absolutely continuous and by \eqref{pam} and \eqref{amp} we obtain
%
that there exist radial
derivatives of $w_{\tau}$ and $w_r$ for a.e. $e^{i\tau}\in \mathbf
T$.
%
By Fatou's theorem (see e.g. \cite[Theorem~6.39]{ABR}, c.f.
\eqref{pam})), we have
\begin{equation}\label{e:REEQ}
\lim_{r\to 1^-}w_{\tau}(re^{i\tau})= F'(\tau)
\end{equation}
for almost every $e^{i\tau}\in \mathbf T$.

Further for a.e. $\tau\in[0,2\pi]$, by using Lagrange theorem we
have
$$\frac {u( e^{i\tau})-u(re^{i\tau})}{1-r}=u_r(p e^{i\tau}),\ \ \
r<p<1$$ and $$\frac {v( e^{i\tau})-v(re^{i\tau})}{1-r}=v_r(q
e^{i\tau}),\ \ \ r<q<1.$$ It follows that for a.e. $\tau\in[0,2\pi]$
\begin{equation}\label{ur} \lim_{r\to 1^-} \frac {u(
e^{i\tau})-u(re^{i\tau})}{1-r}=\lim_{r\to 1^-}u_r(r(e^{i\tau}))
\end{equation}
and
\begin{equation}\label{vr} \lim_{r\to 1^-} \frac
{v( e^{i\tau})-v(re^{i\tau})}{1-r}=\lim_{r\to 1^-}v_r(r(e^{i\tau}))
\end{equation}
and consequently for a.e. $\tau\in[0,2\pi]$
\begin{equation}\label{wr} \lim_{r\to 1^-} \frac {w(
e^{i\tau})-w(re^{i\tau})}{1-r}=\lim_{r\to 1^-}w_r(r(e^{i\tau})).
\end{equation}
 By using the previous facts and the formulas $$w(
e^{i\tau})-w(re^{i\tau})=\int_0^{2\pi}[F(\tau)-F(t)]P(r,\tau-t)dt$$
and \eqref{jaco} we obtain:
\begin{equation}\label{last}
\begin{split}
\lim_{r \to 1^-} J_w(re^{i\tau})&=\lim_{r \to
1^-}\frac{\mathrm{Re}[iw_r(re^{i\tau})\overline{w_\tau(re^{i\tau})}]}{r}\\&=\lim_{r
\to
1^-}\frac{\mathrm{Re}[i(w(e^{i\tau})-w(re^{i\tau}))\overline{w_\tau(re^{i\tau})}]}{(1-r)r}\\&=\lim_{r
\to 1^-}\frac{1}{1-r}\int_0^{2\pi}
P(r,\tau-t)\mathrm{Re}[i(F(\tau)-F(t))\overline{F'(\tau)}]dt\\&=
\lim_{r\to 1^-}\int_{-\pi}^{\pi} K_F(t+\tau,\tau)\frac
{P(r,t)}{1-r}dt,\ \  a.e.\end{split}\end{equation}

where \begin{equation}\label{K} K_F(t,\tau)= f'(\tau)
\text{Re}\,[\overline{(g(f(t))-g(f(\tau)))}\cdot i
g'(f(\tau)))],\end{equation} and $P(r,t)$ is the Poisson kernel. We
refer to \cite[Eq.~5.6]{Rl} for a similar approach, but for some
other purpose.
\\ To continue, observe first that
$$\frac{P(r,t)}{1-r}=\frac {1+r}{2\pi(1+r^2-2r\cos t)}
\leq \frac 1{\pi((1-r)^2+4r\sin^2t/2)}\le \frac{\pi}{4rt^2}$$ for
$0<r<1$ and $t\in [-\pi,\pi]$ because $|\sin(t/2)|\ge t/\pi$.
\\
On the other hand by \eqref{oh} and \eqref{K}, for $$\sigma =
\min\{|f(t+\tau)-f(\tau)|, l-|f(t+\tau)-f(\tau)|\}$$ we obtain
$$\abs{K_F(t+\tau,\tau)}\le
\|F'\|_{\infty}\int_0^{\sigma}\omega(u)du,$$ where $\omega$ is the
modulus of continuity of $g'$. Therefore for $r\ge 1/2$,
\begin{equation}\label{eee}\begin{split}\abs{K_F(t+\tau,\tau)\frac
{P(r,t)}{1-r}}&\le \frac{\|F'\|_{\infty}\pi}{4r
t^2}\int_0^{\sigma}\omega(u)du\\&\le\frac{\sigma}{t}\frac{\|F'\|_{\infty}\pi}{4r
t^2}\int_0^{t}\omega\left(\frac{\sigma}{t}u\right)du\\&\le
\frac{\pi\|F'\|_{\infty}^2}{2}\frac{1}{t^2}\int_0^{t}\omega(\|F'\|_{\infty}u)du:=Q(t).\end{split}\end{equation}
Thus $Q(t)$ is a dominant for the expression
$$\abs{K_F(t+\tau,\tau)\frac {P(r,t)}{1-r}},$$ for $r\ge 1/2$. Having
in mind the equation \eqref{goodd}, we obtain
\[\begin{split}\int_{-\pi}^{\pi} \abs{Q(t)}dt&\le
\frac{2\pi\|F'\|_{\infty}^2}{2}\int_0^{\pi}\frac{1}{t^2}\int_0^{t}\omega(\|F'\|_{\infty}u)du
\\&={\pi\|F'\|_{\infty}^2}\int_0^{\pi}\left(\frac{\omega(\|F'\|_{\infty}u)}{u}-\frac{\omega(\|F'\|_{\infty}u)}{\pi}
\right)du \\&<M<\infty.\end{split}\]

 According to the
Lebesgue Dominated Convergence Theorem, taking the limit under the
integral sign in the last integral in
 \eqref{last} we obtain \eqref{jacfor}.

\end{proof}

 For a
Lipschitz non-decreasing function $f$ and an arc-length
parametrization $g$ of the Dini's smooth curve $\gamma$ we define
the operator $T$ as follows
\begin{equation}\label{Ff}T[f](\tau)=\int_0^{2\pi}
\frac{\mathrm{Re}\,[\overline{(g(f(t))-g(f(\tau)))}\cdot i
g'(f(\tau)))]}{2\sin^2\frac{t-\tau}{2}}dt, \tau\in
[0,2\pi].\end{equation} According to Lemma~\ref{form}, this integral
converges. Notice that if $\gamma$ is a convex Jordan curve then
$\mathrm{Re}\,[\overline{(g(f(t))-g(f(\tau)))}\cdot i
g'(f(\tau))]\ge 0$, and therefore $T[f]>0$. In the next proof, we
will show that under the integral condition $T[f]>0$ the harmonic
extension of a bi-Lipschitz mapping is a diffeomorphism regardless
of the condition of convexity.

\begin{proof}[Proof of Theorem~\ref{ma}]
Assume for simplicity that $|\gamma|= 2\pi$. The general case
follows by normalization. Let $g:[0,2\pi]\to \gamma$ be an arc
length parametrization of $\gamma$. Then $F(e^{it})=g(f(t))$, where
$f:\mathbf R\to \mathbf R$ is a Lipschitz weak homeomorphism such
that $f(t+2\pi) = f(t) + 2\pi$. From \eqref{Ff} we have
\[\begin{split}T[f](\tau)&= \lim_{\epsilon\to
0^+}\int_{\epsilon}^{\pi}
\frac{\text{Re}\,[\overline{(g(f(t+\tau))-g(f(\tau)))}\cdot i
g'(f(\tau)))]}{2\sin^2\frac{t}{2}}\frac{dt}{2\pi}
\\&+\lim_{\epsilon\to
0^+}\int_{-\pi}^{-\epsilon}
\frac{\text{Re}\,[\overline{(g(f(t+\tau))-g(f(\tau)))}\cdot i
g'(f(\tau)))]}{2\sin^2\frac{t}{2}}\frac{dt}{2\pi} .\end{split}\]
Assume that $\beta:[0,2\pi]\to \mathbf R$ is a continuous functions
such that
\begin{equation}\label{gprim}g'(s) = e^{i\beta(s)},\ \ \beta(0) =\beta(2\pi).\end{equation} Then
\begin{equation}\label{uma}|g'(s)-g'(t)|=
2\abs{\sin\frac{\beta(t)-\beta(s)}{2}}.\end{equation} Let
$\omega_\beta$ be the modulus of continuity of $g'$. Then
\begin{equation}\label{david1}\omega_\beta(\rho) = \max_{|t-s|\le
\rho}2\abs{\sin\frac{\beta(t)-\beta(s)}{2}}.\end{equation} Since
$\gamma\in C^{1,\alpha}$,
\begin{equation}\label{david}\omega_\beta(\rho)\le c(\gamma)
\rho^\alpha.\end{equation} Further from \eqref{gprim}, we have
\[\begin{split}\frac{\text{Re}\,[\overline{(g(f(t+\tau))-g(f(\tau)))}\cdot i
g'(f(\tau))]}{2\sin^2\frac{t}{2}}
&=\frac{\text{Re}\,[\overline{\int_{f(\tau)}^{f(t+\tau)}g'(s)ds}\cdot
i g'(f(\tau))]}{2\sin^2\frac{t}{2}}\\
&=\frac{\text{Re}\,[\overline{\int_{f(\tau)}^{f(t+\tau)}e^{i\beta(s)}ds}\cdot
i e^{i\beta(\tau)}]}{2\sin^2\frac{t}{2}}
\\&=\frac{-\text{Im}\,[\overline{\int_{f(\tau)}^{f(t+\tau)}e^{i\beta(s)-\beta(\tau)}ds}]}{2\sin^2\frac{t}{2}} \\&=
\frac{\int_{f(\tau)}^{f(t+\tau)}\sin[\beta(s)-\beta(f(\tau))]ds}{2\sin^2\frac
t2}.\end{split}\] Taking $$dU = \frac{1}{2\sin^2\frac t2}dt\text{ \
and \ \ } V
=\int_{f(\tau)}^{f(t+\tau)}\sin[\beta(s)-\beta(f(\tau))]ds,$$ we
obtain  that $$U = -\cot\frac t2 \ \ \text{and}  \ \ dV =
f'(t+\tau)\sin[\beta(f(t+\tau))-\beta(f(\tau))]dt.$$ To continue
recall that $f$ is Lipschitz with a Lipschitz constant $L$. Thus
\[\begin{split}|\lim_{\epsilon\to 0^+}U(t)V(t)|_{\epsilon}^\pi|
&=\left|\lim_{\epsilon\to 0^+} \cot \frac{\epsilon}{2}
\int_{f(\tau)}^{f(\epsilon+\tau)}\sin[\beta(s)-\beta(f(\tau))]ds
\right|\\&\le\lim_{\epsilon\to 0^+} \cot
\frac{\epsilon}{2}|\sin[\beta(\epsilon+\tau)-\beta(f(\tau))]|
|f(\epsilon+\tau)-f(\tau)|\\& \le\lim_{\epsilon\to 0^+}
L\epsilon\cot \frac{\epsilon}{2}
\omega_\beta(\epsilon)=0.\end{split}\]
 Similarly we have
$$\lim_{\epsilon\to 0^+}U(t)V(t)|_{-\pi}^{-\epsilon} =0.$$  By a partial integration we obtain
\[\begin{split} T[f](\tau)&= \lim_{\epsilon\to 0^+} \left(UV|_{\epsilon}^\pi +\int_\epsilon^\pi  f'(t+\tau)\cdot
\sin[\beta(f(t+\tau))-\beta(f(\tau))]\cot \frac t2
\frac{dt}{2\pi}\right)\\&+\lim_{\epsilon\to 0^+}
\left(UV|_{-\pi}^{-\epsilon} + \int_{-\pi}^{-\epsilon}
f'(t+\tau)\cdot \sin[\beta(f(t+\tau))-\beta(f(\tau))]\cot \frac t2
\frac{dt}{2\pi}\right)\\&=\int_{-\pi}^{\pi}f'(t+\tau)\cdot
\sin[\beta(f(t+\tau))-\beta(f(\tau))]\cot \frac t2
\frac{dt}{2\pi}.\end{split}\]  Hence
$$T[f](\tau)=\int_{-\pi}^{\pi}f'(t+\tau)\cdot
\sin[\beta(f(t+\tau))-\beta(f(\tau))]\cot \frac t2
\frac{dt}{2\pi}.$$ By using Lemma~\ref{dri}, we can choose a family
of diffeomorphisms $f_n$ converging uniformly to $f$. Then
$$T[f_n](\tau)=\int_{-\pi}^{\pi}f_n'(t+\tau)\cdot
\sin[\beta(f_n(t+\tau))-\beta(f_n(\tau))]\cot \frac t2
\frac{dt}{2\pi}.$$ We are going to show that $T[f_n]$ converges
uniformly to $T[f]$. In order to do this, we apply Arzela-Ascoli
theorem.

First of all \[\begin{split}|T[f_n](\tau)|&\le
\frac{1}{\pi}\|f_n'\|_{\infty}\int_{0}^{\pi}\omega_\beta(\|f_n'\|_\infty
t)\cot \frac t2 dt
\\&\le \frac{1}{\pi}\|f'\|_{\infty}\int_{0}^{\pi}\omega_\beta(\|f'\|_\infty t)\cot \frac t2 dt
= C(f,\gamma)<\infty.\end{split}\] We prove now that $T[f_n]$ is
equicontinuous family of functions. We have to estimate the
quantity: $$|T[f_n](\tau) - T[f_n](\tau_0)|.$$ Assume without loss
of generality that $\tau_0 = 0 $. Then
\[\begin{split}|T[f_n](\tau) -
T[f_n](0)|&=\left|\int_{-\pi}^{\pi}f_n'(t+\tau)\cdot
\sin[\beta(f_n(t+\tau))-\beta(f_n(\tau))]\cot \frac t2
\frac{dt}{2\pi}\right.\\&-\left.\int_{-\pi}^{\pi}f_n'(t)\cdot
\sin[\beta(f_n(t))-\beta(f_n(0))]\cot \frac t2
\frac{dt}{2\pi}\right|\le A + B,\end{split}\] where $$ A =
\left|\int_{-\pi}^{\pi}(f_n'(t+\tau)-f_n'(t))\cdot
\sin[\beta(f_n(t+\tau))-\beta(f_n(\tau))]\cot \frac t2
\frac{dt}{2\pi}\right|$$ and $$ B =
\left|\int_{-\pi}^{\pi}f_n'(t)\cdot
\{\sin[\beta(f_n(t))-\beta(f_n(0))]-\sin[\beta(f_n(t+\tau))-\beta(f_n(\tau))]\}\cot
\frac t2 \frac{dt}{2\pi}\right|.$$ Take $r\ge 1$, $p>1$, $q>1$ such
that $$\frac 1p + \frac 1q = 1,$$and $\delta\in (0,1)$.

 In what follows, for a function $g\in L^p(\mathbf
T)$ we have in mind the following $p-$norm:
$$\|g\|_p =
\left(\int_0^{2\pi}\abs{g(e^{it})}^p\frac{dt}{2\pi}\right)^{1/p}.$$
Define $f_\tau(x):=f(x+\tau)$. By \eqref{aar} we have
$$f_{n}=\frac{
n}{n+1}\left(f\ast\rho_{\varepsilon_n}+\frac{x}{n}\right).$$ Thus
\begin{equation}\label{varas}|f'_{n,\tau}-f'_n|=\frac{n}{n+1}|(f'_\tau-f')\ast \rho_{\varepsilon_n}|.\end{equation} According to Young's inequality for convolution
(\cite[pp. 54-55; 8, Theorem 20.18]{young}), we obtain that
$$\|(f'_\tau-f')\ast \rho_{\varepsilon_n}\|_r\le \|f'_{\tau} - f'\|_r.$$ In view of
\eqref{david} and \eqref{varas}, for $1<q<\frac{1}{1-\alpha}$,
according to H\"older inequality we have
\[\begin{split}A &\le \|f_n'(t+\tau)-f_n'(t)\|_p\cdot \|
\sin[\beta(f_n(t+\tau))-\beta(f_n(\tau))]\cot \frac t2
\frac{1}{2\pi}\|_q\\&\le \|f'(t+\tau)-f'(t)\|_p\cdot \|
\omega_\beta(|f_n|_\infty t)\cot \frac t2 \frac{1}{2\pi}\|_q \\&\le
C_1(\gamma)\|f'\|_\infty\|f'(t+\tau)-f'(t)\|_p,\end{split}\] i.e.
\begin{equation}\label{AA}
A \le C_1(\gamma)\|f'\|_\infty\|f'(t+\tau)-f'(t)\|_p.
\end{equation}
Let now estimate $B$. First of all \begin{equation}\label{beb} B \le
\|f'\|_\infty
\|\{\sin[\beta(f_n(t))-\beta(f_n(0))]-\sin[\beta(f_n(t+\tau))-\beta(f_n(\tau))]\}\cot
\frac t2\|_1.
\end{equation}
On the other hand, using again H\"older inequality we have
\[\begin{split}\|\{\sin&[\beta(f_n(t))-\beta(f_n(0))]-\sin[\beta(f_n(t+\tau))-\beta(f_n(\tau))]\}\cot
\frac t2\|_1\\&\le
\|\{\sin[\beta(f_n(t))-\beta(f_n(0))]-\sin[\beta(f_n(t+\tau))-\beta(f_n(\tau))]\}^\delta\|_p\\&
\times\|\{\sin[\beta(f_n(t))-\beta(f_n(0))]-\sin[\beta(f_n(t+\tau))-\beta(f_n(\tau))]\}^{1-\delta}\cot
\frac t2\|_q.
\end{split}
\]
Further
\[\begin{split}\|\{\sin&[\beta(f_n(t))-\beta(f_n(0))]-\sin[\beta(f_n(t+\tau))-\beta(f_n(\tau))]\}^\delta\|_p
\\&\le\|\{\abs{2\sin\frac{\beta(f_n(t))-\beta(f_n(0))-\beta(f_n(t+\tau))+\beta(f_n(\tau))}{2}}\}^\delta\|_p\\&\le
\|\{\abs{2\sin\frac{\beta(f_n(t+\tau))-\beta(f_n(t))}2}\}^\delta\|_p\\&+
\|\{\abs{2\sin\frac{\beta(f_n(\tau))-\beta(f_n(0))}2}\}^\delta\|_p\\&\le
\omega_\beta(|f_n'|_{\infty}\tau)^\delta+\omega_\beta(|f_n'|_{\infty}\tau)^\delta
= 2\omega_\beta(|f_n'|_{\infty}\tau)^\delta\le
2\omega_\beta(|f'|_\infty\tau)^\delta,
\end{split}
\]
and
\[\begin{split}\|\{\sin[\beta(f_n(t))&-\beta(f_n(0))]-\sin[\beta(f_n(t+\tau))-\beta(f_n(\tau))]\}^{1-\delta}\cot
\frac t2\|_q\\&\le
\|(2\omega_\beta(|f_n'|_{\infty}t)^{1-\delta}\cot\frac
t2\|_q.\end{split}\] Choose $q$ and $\delta$ such that
$$(\alpha - \alpha \delta - 1) q>-1.$$
Then the integral
$$\|2\omega_\beta(|f_n'|_{\infty}t)^{1-\delta}\cot\frac t2\|_q$$
converges and it is less or equal to $$
C(\gamma)\|f_n'\|_{\infty}^{1-\delta}\le
C(\gamma)\|f'\|_{\infty}^{1-\delta}.  $$ Therefore
\begin{equation}\label{BB}B \le 2\|f'\|_\infty
C(\gamma)\|f'\|_{\infty}^{1-\delta}\omega_\beta(\|f'\|_\infty\tau)^\delta.\end{equation}
Since a translation is continuous (see \cite[Theorem~9.5]{stein}),
\eqref{AA} and \eqref{BB} imply that the family $\{T[f_n]\}$ is
equicontinuous.  By Arzela-Ascoli theorem it follows that
$$\lim_{n\to \infty} \|T[f_n] - T[f]\|_{\infty}=0.$$ Thus $T[f]$ is
continuous.
\\
Moreover, since $f_n$ is a diffeomorphism, for $n$ sufficiently
large there holds the following inequality
$$J_{w_n}(e^{i\tau})=f_n'(\tau)T[f_n](e^{i\tau})>0, e^{i\tau}\in
\mathbf T.$$ Since $f_n\in C^{\infty}$, it follows that
$$w_n = P[F_n]\in C^1(\overline{\mathbf U}).$$ Therefore all the
conditions of Proposition~\ref{ales} are satisfied. This means that
$w_n$ is a harmonic diffeomorphism of the unit disk onto the domain
$D$.

Since, by a result of Hengartner and Schober \cite{hsc}, the limit
function $w$ of a locally uniformly convergent sequence of univalent
harmonic mappings $w_n$ on $\mathbf U$ is either univalent on
$\mathbf U$, is a constant, or its image lies on a straight-line, we
obtain that $w=P[F]$ is univalent. The proof is completed.
\end{proof}

\begin{remark}\label{rere} If $\gamma$ is a $C^{1,\alpha}$ convex curve, then
$\mathrm{Re}\,[\overline{(g(f(t))-g(f(\tau)))}\cdot i
g'(f(\tau))]\ge 0$ and therefore  $T[f](\tau)>0$. By the proof of
Theorem~\ref{ma}, $\tau\to T[f](\tau)$ is continuous. Therefore
$\min_{\tau\in[0,2\pi]}T[f](\tau)=\delta>0$.

\end{remark}
\section{Quasiconformal harmonic mappings}
An injective harmonic mapping $w=u+iv$, is called
$K$-\emph{quasiconformal} ($K$-q.c), $K\ge 1$, if
\begin{equation}\label{qc}|w_{\bar z}|\le k|w_z|\end{equation} on $D$ where $k={(K-1)}/{(K+1)}$.
Here
$$w_z := \frac{1}{2}\left(w_x-iw_y\right)\text{ and } w_{\bar z} := \frac{1}{2}\left(w_x+i
w_y\right).$$ Notice that, since if $|\nabla w(z)|:=\max\{|\nabla
w(z) h|: |h|=1\}=|w_z|+|w_{\bar z}|$ $l(\nabla w(z)):=\min\{|\nabla
w(z) h|: |h|=1\}=||w_z|-|w_{\bar z}||$, the condition \eqref{qc} is
equivalent with \begin{equation} \label{1ll}|\nabla w(z)|\le K
l(\nabla w(z)).
\end{equation} For a general definition of quasiregular mappings and quasiconformal mappings we
refer to the book of Ahlfors \cite{Ahl}. In this section we apply
Theorem~\ref{ma} to the class of q.c. harmonic mappings. The area of
quasiconformal harmonic mappings is very active
 area of research. For a background on this theory we
refer \cite{hs}, \cite{trans}-\cite{kojic}, \cite{om}, \cite{pk},
\cite{mp} and \cite{wan}. In this section we obtain some new results
concerning a characterization of this class. We will restrict
ourselves to the class of q.c. harmonic mappings $w$ between the
unit disk $\mathbf U$ and a Jordan domain $D$. The unit disk is
taken because of simplicity. Namely, if $w:\Omega\to D$ is q.c.
harmonic, and $a:\mathbf U\to \Omega$ is conformal, then $w\circ a$,
is also q.c. harmonic. However the image domain $D$ cannot be
replaced by the unit disk.

The case when $D$ is a convex domain is treated in detail by the
author and others in above cited papers. In this section we will use
our main result to yield a characterization of quasiconformal
harmonic mappings onto a Jordan that is not necessarily convex in
terms of boundary data.

%
To state the main result of this section, we make use of Hilbert
transforms formalism. It provides a necessary and a sufficient
condition for the harmonic extension of a homeomorphism from the
unit circle to a $C^{2}$ Jordan curve $\gamma$ to be a q.c mapping.
It is an extension of the corresponding result
\cite[Theorem~3.1]{mathz} related to convex Jordan domains.
\begin{theorem}\label{convex1} Let $F:\mathbf T\to \gamma$ be a sense preserving
homeomorphism of the unit circle onto the Jordan curve
$\gamma=\partial D \in C^{2}$. Then $w=P[F]$ is a quasiconformal
mapping of the unit disk onto $D$ if and only if $F$ is absolutely
continuous and
\begin{equation}\label{1first}0<l(F):=\mathrm{ess\,inf\,} l(\nabla w(e^{i\tau})),\end{equation}
\begin{equation}\label{1second}\|F'\|_\infty:= \mathrm{ess\,sup\,}|F'(\tau)|<\infty\end{equation} and
\begin{equation}\label{1third}\|H(F')\|_\infty:=\mathrm{ess}\sup |H(F')(\tau)|<\infty.\end{equation}
If $F$ satisfies the conditions \eqref{1first}, \eqref{1second} and
\eqref{1third}, then $w=P[F]$ is $K$ quasiconformal, where
\begin{equation}\label{KK1}K:=\frac{\sqrt{\|F'\|^2_\infty +
\|H(F')\|^2_\infty-l(F)^2}}{l(F)}.\end{equation} The constant $K$ is
approximately sharp for small values of $K$: if $w$ is the identity
or if it is a mapping close to the identity, then $K=1$ or $K$ is
close to $1$ (respectively).
\end{theorem}
\begin{proof}[{\bf The proof of necessity.}] Suppose that $w=P[F]=g+\overline h$ is a $K-$q.c. harmonic
mapping that satisfies the conditions of the theorem. By
\cite[Theorem~2.1]{mathz}, we see that $w$ is Lipschitz continuous,
\begin{equation}L:=\label{1be}\|F'\|_{\infty}<\infty\end{equation} and
\begin{equation}\label{kl}
|\nabla w(z)|\le KL.
\end{equation}  By \cite[Theorem~1.4]{kalajpisa}  we have for
$b=w(0)$
\begin{equation}\label{1equqc} |\partial w(z)|-|\bar\partial
w(z)|\geq C(\Omega, K,b)>0,\, z\in \mathbf U.
\end{equation} Because of \eqref{kl}, the analytic functions $\partial
w(z)$ and  $\bar \partial w(z)$ are bounded, and therefore by
Fatou's theorem:
\begin{equation}\label{1beo}\lim_{r\to 1^-}(|\partial
w(re^{i\tau})|-|\bar\partial w(re^{i\tau})|)= |\partial
w(e^{i\tau})|-|\bar\partial w(e^{i\tau})|\ \ a.e.\end{equation}
Combining \eqref{1be}, \eqref{1beo} and \eqref{1equqc}, we get
\eqref{1first} and \eqref{1second}.

Next we prove \eqref{1third}. Observe first that $
w_r=e^{i\tau}w_z+e^{-i\tau}w_{\overline z}.$ Thus
\begin{equation}\label{1derrad}|w_r|\le |\nabla w|\le
KL.\end{equation} Therefore $rw_r=P[H(F')]$ is a bounded harmonic
function which implies that $H(F')\in L^\infty(\mathbf T)$.
Therefore \eqref{1third} holds and the necessity proof is completed.
\\
{\bf The proof of sufficiency.} We have to prove that under the
conditions \eqref{1first},
 \eqref{1second} and
\eqref{1third} $w$ is quasiconformal.
 From
$$0<l(F):=\mathrm{ess\,inf\,} l(\nabla
w(e^{i\tau}))$$ we obtain that
$$J_w(e^{i\tau})= (|w_z|+|w_{\bar z}|)l(\nabla w(e^{i\tau}))\ge l(F)^2\ \  (a.e.)$$

As $rw_r$ is a harmonic conjugate of $w_\tau$, it turns out that if
$F$ is absolutely continuous, then
\begin{equation}\label{1amp}\lim_{r\to 1^-}w_r(re^{i\tau}) =
H(F')(\tau)\,\, (a.e.),\end{equation} and
\begin{equation}\label{1pam}\lim_{r\to 1^-} w_\tau(re^{i\tau}) =
F'(\tau)\  \ (a.e.).\end{equation} As $$|w_z|^2+|w_{\bar z}|^2 =
\frac 12 \left( |w_r|^2+ \frac{|w_\tau|^2}{r^2}\right),$$ it follows
that for a.e. $\tau\in[0,2\pi)$
\begin{equation}\label{1boundir}
\lim_{r\to 1^-}|w_z(re^{i\tau})|^2+|w_{\bar
z}(re^{i\tau})|^2=|w_z(e^{i\tau})|^2+|w_{\bar
z}(e^{i\tau})|^2\le\frac{1}{2}(\|F'\|^2_\infty +
\|H(F')\|^2_\infty).
\end{equation}
  To continue we make use of
\eqref{1first}. From \eqref{1boundir}, \eqref{1first} and
\eqref{1ll}, for a.e. $\tau\in[0,2\pi)$,
\begin{equation}\label{1ldef}\frac{|w_z(e^{i\tau})|^2+|w_{\bar z}(e^{i\tau})|^2}
{(|w_z(e^{i\tau})|-|w_{\bar z}(e^{i\tau})|)^2}\le
\frac{\|F'\|^2_\infty + \|H(F')\|^2_\infty}{2l(F)^2}.\end{equation}
Hence
\begin{equation}\label{1from}|w_z(e^{i\tau})|^2+|w_{\bar z}(e^{i\tau})|^2 \le S
(|w_z(e^{i\tau})|-|w_{\bar z}(e^{i\tau})|)^2 \ \ \
(a.e.),\end{equation} where
\begin{equation}\label{1S}
S:=\frac{\|F'\|^2_\infty + \|H(F')\|^2_\infty}{2l(F)^2}.
\end{equation}
According to \eqref{1ldef}, $S\ge 1$. Let $$\mu(e^{i\tau}):=
\left|\frac{w_{\bar z}(e^{i\tau})}{w_z(e^{i\tau})}\right|.$$ Since
every $C^{2}$ curve is $C^{1,\alpha}$ curve, Theorem~\ref{ma} shows
that $w=g+\overline k$ is univalent and according to Lewy's theorem
$J_w(z)=|g'(z)|^2-|h'(z)|^2>0$. Thus $a(z)= \overline{w_{\bar
z}}/w_z=h'/g'$ is an analytic function bounded by 1. As
$\mu(e^{i\tau})=|a(e^{i\tau})|$, we have  $\mu(e^{i\tau})\le 1$.
Then \eqref{1from} can be written as $$ 1+\mu^2(e^{i\tau})\le
S(1-\mu(e^{i\tau}))^2,
$$ i.e. if $S=1$, then $\mu(e^{i\tau})=0$ a.e. and if $S>1$, then
\begin{equation}\label{1qineq}
\mu^2(S-1) -2\mu S + S - 1=(S-1)(\mu-\mu_1)(\mu-\mu_2)\ge 0,
\end{equation}
where $$\mu_1 = \frac{S+\sqrt{2S-1}}{S-1}$$ and $$\mu_2  =
\frac{S-1}{S+\sqrt{2S-1}}.$$ If $S>1$, then from \eqref{1qineq} it
follows that $\mu(e^{i\tau})\le \mu_2$ or $\mu(e^{i\tau}) \ge
\mu_1$. But $\mu(e^{i\tau}) \le 1$ and therefore

\begin{equation}\label{1mu}
\mu(e^{i\tau}) \le \frac{S-1}{S+\sqrt{2S-1}} \ \ \ \ (a.e.).
\end{equation}

If $S =1$, then \eqref{1mu} clearly holds. Define $\mu(z)=|a(z)|$.
Since $a$ is a bounded analytic function, by the maximum principle
it follows that
$$\mu (z) \le k:=\mu_2, $$ for $z\in \mathbf U$.
This yields that $$K(z) \le K := \frac{1+k}{1-k}=
\frac{2S-1+\sqrt{2S-1}}{\sqrt{2S-1}+1}=\sqrt{2S-1} ,$$ i.e.

$$K(z) \le \frac{\sqrt{\|F'\|^2_\infty +
\|H(F')\|^2_\infty-l(F)^2}}{l(F)}$$ which means that $w$ is $K
=\frac{\sqrt{\|F'\|^2_\infty +
\|H(F')\|^2_\infty-l(F)^2}}{l(F)}-$quasiconfomal. The result is
asymptotically sharp because $K=1$ for $w$ being the identity. This
finishes the proof of Theorem~\ref{KK1}.\end{proof}

\subsection*{A conjecture} Let $F: \mathbf T\to \gamma\subset \mathbf C$ be a homeomorphism
of bounded variation, where $\gamma$ is Dini smooth. Let $D$ be the
bounded domain such that $\partial D =\gamma$. The mapping $w=P[F]$
is a diffeomorphism of $\mathbf U$ onto $D$ if and only if
\begin{equation}\label{use15}
\mathrm{ess\,inf}\{J_w(e^{it}):t\in[0,2\pi]\} \ge 0.\end{equation}

\subsection*{Acknowledgment} I am thankful to the referee for providing very constructive comments and
help in improving the contents of this paper.

\end{document}